\documentclass[11pt]{amsart}
\usepackage[T1]{fontenc}
\usepackage{lmodern}
\usepackage{amsmath,amssymb,amsthm,mathrsfs,comment}
\usepackage{tikz}
\usetikzlibrary{matrix,arrows,decorations.pathmorphing,cd,positioning,calc}
\usepackage{enumitem}
\usepackage{hyperref}
\hypersetup{colorlinks=true,linkcolor=blue,citecolor=blue,urlcolor=blue,
 pdftitle={End-Semidualizing Modules},pdfauthor={Luigi Ferraro and Justin Lyle}}
\usepackage{cleveref}
\usepackage{microtype}
\numberwithin{equation}{section}
\newtheorem{theorem}{Theorem}[section]
\newtheorem{lem}[theorem]{Lemma}
\newtheorem{prop}[theorem]{Proposition}
\newtheorem{proposition}[theorem]{Proposition}
\newtheorem{cor}[theorem]{Corollary}
\newtheorem{corollary}[theorem]{Corollary}
\newtheorem*{introtheorem}{Theorem}
\theoremstyle{definition}
\newtheorem{defn}[theorem]{Definition}
\newtheorem{definition}[theorem]{Definition}
\theoremstyle{remark}
\newtheorem{remark}[theorem]{Remark}
\newtheorem{example}[theorem]{Example}
\newtheorem{chunk}[theorem]{}
\crefname{lem}{lemma}{lemmas}
\Crefname{lem}{Lemma}{Lemmas}
\crefname{prop}{proposition}{propositions}
\Crefname{prop}{Proposition}{Propositions}
\crefname{proposition}{proposition}{propositions}
\Crefname{proposition}{Proposition}{Propositions}
\crefname{cor}{corollary}{corollaries}
\Crefname{cor}{Corollary}{Corollaries}
\crefname{corollary}{corollary}{corollaries}
\Crefname{corollary}{Corollary}{Corollaries}
\crefname{defn}{definition}{definitions}
\Crefname{defn}{Definition}{Definitions}
\crefname{definition}{definition}{definitions}
\Crefname{definition}{Definition}{Definitions}
\DeclareMathOperator{\End}{End}
\DeclareMathOperator{\Hom}{Hom}
\DeclareMathOperator{\Ext}{Ext}
\DeclareMathOperator{\Tor}{Tor}
\DeclareMathOperator{\Spec}{Spec}
\DeclareMathOperator{\Max}{Max}
\DeclareMathOperator{\Supp}{Supp}

\DeclareMathOperator{\depth}{depth}
\DeclareMathOperator{\fd}{fd}

\DeclareMathOperator{\id}{id}

\DeclareMathOperator{\coker}{coker}

\newcommand{\A}{\mathcal A}
\newcommand{\B}{\mathcal B}
\newcommand{\m}{\mathfrak m}
\newcommand{\p}{\mathfrak p}

\begin{document}
\title[On End-Semidualizing Modules]{On End-Semidualizing Modules}
\author[Ferraro]{Luigi Ferraro}
\address[Ferraro]{School of Mathematical and Statistical Sciences \\ University of Texas Rio Grande Valley \\ Edinburg, TX 78539, U.S.A}
\email[Luigi Ferraro]{luigi.ferraro@utrgv.edu}
\urladdr{https://faculty.utrgv.edu/luigi.ferraro/}
\author[Lyle]{Justin Lyle}
\address[Lyle]{Department of Mathematics and Statistics \\ 305 W Samford Avenue \\
Auburn University, AL 36849}
\email[Justin Lyle]{jll0107@auburn.edu}
\urladdr{https://jlyle42.github.io/justinlyle/}
\subjclass[2020]{Primary 16E30; Secondary 13D07, 13C14, 16D90}
\keywords{End-semidualizing module, Auslander class, Bass class, Foxby equivalence, faithfully balanced module, depth}
\begin{abstract}
We introduce and study the notion of End-semidualizing modules which are generalizations of semidualizing modules in the classical sense, and also extend the notion of $S-R$ semidualizing bimodules studied by Holm-White and Wakamatsu. By constrast with these notions, End-semidualizing modules exist abundantly. We characterize these modules by the inclusion of projective and flat modules
in their Auslander class, or injective modules in their Bass class, and give criteria for when these Foxby classes satisfy the
two-out-of-three property. Over commutative Noetherian rings, we prove base-change and duality results,
and characterize local depth inequalities by the inclusion of all modules of
finite flat or injective dimension in the corresponding classes. 
\end{abstract}
\maketitle

\section{Introduction}\label{sec:intro}

Let $R$ be a commutative Noetherian local ring. A finitely generated $R$-module $N$ is said to be \emph{semidualizing} if $\Ext^{>0}_R(N,N)=0$ and the natural homothety map $R\to\End_R(N)$ is an isomorphism. Semidualizing modules fill the gaps between free modules of rank $1$ and dualizing modules. Indeed, if $R$ is Gorenstein then the only semidualizing $R$-module, up to isomorphism, is the free module of rank $1$. Conversely, if $R$ admits a dualizing module and $R$ is its only semidualizing module, then $R$ is Gorenstein; see \cite{LarsBook}.

An obvious drawback to this definition is that it does not directly extend to the noncommutative setting. Indeed, if $R$ is noncommutative, multiplication by an element of $R$ need not define an $R$-linear endomorphism of a left $R$-module. To remedy this, one works with $S-R$ semidualizing bimodules, as developed by Holm-White \cite{HW07}. In the conventions of this paper, a finitely generated left $R$-module $N$ carries a commuting left action of $E:=\End_R(N)$. Subject to degreewise finite projective resolutions on both sides, the corresponding $S-R$ semidualizing bimodule conditions are that
\[
 \Ext_R^{>0}(N,N)=\Ext_E^{>0}(N,N)=0,
\]
and that the natural homothety map $R \to \End_E(N)$ is an isomorphism. These conditions are closely related to Wakamatsu's earlier work on modules with trivial self-extensions \cite{W88}.

A semidualizing module $N$ in the classical sense is often studied via its so-called Foxby classes, namely its Auslander class $\A_N(R)$ and its Bass class $\B_N(R)$, which are categories of modules on which Hom and tensor product functors associated to $N$ have good behavior. See Definitions \ref{def:auslander} and \ref{def:bass} for the precise definitions we work with, which do not require $N$ to be semidualizing. Holm-White show that many of the familiar behaviors of semidualizing modules extend when adapted appropriately to semidualizing bimodules.

In this work we relax the notion of a semidualizing bimodule by removing the requirement that $\Ext_R^{>0}(N,N)=0$ to produce what we call \emph{End-semidualizing modules}. Dropping this requirement on the vanishing of Ext produces a substantial generalization. Indeed, any finitely generated generator (or cogenerator), and in particular any module of the form $N:=M\oplus R$, is End-semidualizing. In this example the extra self-orthogonality condition is precisely
\[
 \Ext_R^{>0}(M,M\oplus R)=0.
\]
Thus the additional requirement concerns the Auslander-Reiten condition, whereas the End-semidualizing condition holds for every $M$, without such a vanishing hypothesis. We also give an example of an $\End$-semidualizing module that is not a generator, cogenerator, nor semidualizing in the classical sense (see \Cref{neithergeneratornorcogenerator}).

As additional motivation for the study of the End-semidualizing condition, our first main result demonstrates its naturality. Throughout the noncommutative results, $R$ is Noetherian on both sides, $N$ is finitely generated over $R$, and $N$ admits a resolution by finitely generated projective left $E$-modules. The last hypothesis is automatic when $R$ is module-finite over its center, and also when $N$ is a generator. 

\begin{introtheorem}
The following are equivalent:
\begin{enumerate}[label=$(\arabic*)$]
\item $N$ is End-semidualizing.
\item Every projective right $R$-module belongs to $\A_N(R)$.
\item Every flat right $R$-module belongs to $\A_N(R)$.
\item Every injective left $R$-module belongs to $\B_N(R)$.
\end{enumerate}
\end{introtheorem}

This is proved in Propositions \ref{projectiveex} and \ref{injectiveex}. 

A powerful property of semidualizing modules in the classical sense is that their Foxby classes satisfy the so-called two-out-of-three property. That is, if $0\to A\to B\to C\to0$ is a short exact sequence for which two of the terms lie in the relevant Auslander class (resp.\ Bass class), then so does the third. We show that while the Auslander and Bass classes of any module always satisfy appropriate weakenings of this condition, the full two-out-of-three property need not hold for an End-semidualizing module. The relevant additional condition is that $-\otimes_E N$ \emph{detect zero objects}, i.e., that $X\otimes_E N=0$ implies $X=0$.

\begin{introtheorem}
If $-\otimes_E N$ detects zero objects, then both $\A_N(R)$ and $\B_N(R)$ satisfy the two-out-of-three property. Consequently, if $N$ is End-semidualizing, these classes contain, respectively, all right $R$-modules of finite flat dimension and all left $R$-modules of finite injective dimension.
\end{introtheorem}

The first assertion does not require the End-semidualizing condition. Detection of zero objects is automatic when $R$ is an Artin algebra and $N$ is indecomposable. We give an explicit example over a triangular matrix algebra showing that the two-out-of-three property can fail when $N$ is decomposable. Independently of the End-semidualizing condition, character duality shows that $\A_N(R)$ has the two-out-of-three property if and only if $\B_N(R)$ does; see Proposition~\ref{character-foxby-duality}.

For the commutative case in particular, we provide several sufficient criteria under which the two-out-of-three property holds (\Cref{2outof3commutative}), and we show that the two-out-of-three property is closely tied to the depth of an End-semidualizing module $N$. We call $N$ \emph{locally deep} if $\depth_{R_\p}N_\p\geq\depth R_\p$ for every prime ideal $\p$. The following result is Theorem \ref{pdmcm}.

\begin{introtheorem}
Let $R$ be a commutative Noetherian ring and let $N$ be End-semidualizing. The following are equivalent:
\begin{enumerate}[label=$(\arabic*)$]
\item $N$ is locally deep.
\item Every $R$-module of finite flat dimension belongs to $\A_N(R)$.
\item Every $R$-module of finite injective dimension belongs to $\B_N(R)$.
\end{enumerate}
\end{introtheorem}

We now outline the structure of this work. Section~\ref{sec:prelim} establishes the conventions, basic examples, and zero-detection lemma. Sections~\ref{sec:auslander} and~\ref{sec:bass} study the Auslander and Bass classes and their exact-sequence properties. Section~\ref{sec:foxby} proves Foxby equivalence and its recognition criteria. Section~\ref{sec:commutative} treats flat base change, localization, duality, and depth over commutative rings, and finishes by proving descent of the End-semidualizing condition along regular sequences under finite self-extension vanishing.

\section{Preliminaries}\label{sec:prelim}

Throughout, $R$ is a possibly noncommutative, two-sided Noetherian ring and
$N$ is a finitely generated left $R$-module. We set $E:=\End_R(N)$, with
multiplication given by composition. We assume that $N$, regarded as a left
$E$-module, admits a projective resolution consisting of finitely generated
projective modules. This finiteness hypothesis will be used when commuting
$\Hom$ and $\Ext$ with direct limits and when applying evaluation
isomorphisms. It holds automatically if $R$ is module-finite over a
commutative Noetherian central subring, and also if $N$ is a generator, as
shown below. Modules other than the fixed module $N$ need not be finitely
generated. Their sides are specified in each statement.

We note that $N$ has a natural left module structure over $E$, given by
$f\cdot x=f(x)$, and that the left $R$- and $E$-actions commute.
Consequently, if $M$ is a right $R$-module, then $M\otimes_RN$ is a left
$E$-module via $f(m\otimes n)=m\otimes f(n)$. If $X$ is a left
$E$-module, then $\Hom_E(N,X)$ is a right $R$-module via
$(h r)(n)=h(rn)$. Similarly, for a left $R$-module $Y$, the module
$\Hom_R(N,Y)$ is a right $E$-module via $h f=h\circ f$, and for a right
$E$-module $L$, the module $L\otimes_EN$ is a left $R$-module via
$r(\ell\otimes n)=\ell\otimes rn$. These conventions account for all
tensor products and endomorphism rings used below.

We say that $N$ is a \emph{generator} if $R$ is a direct summand of
$N^{\oplus t}$ for some positive integer $t$.

\begin{chunk}
There is a natural ring homomorphism
\[
\iota_N:R\longrightarrow\End_E(N),\qquad
\iota_N(r)(n)=rn.
\]
\begin{enumerate}
\item[$(1)$] We say $N$ is \emph{faithful} if $\iota_N$ is injective.
\item[$(2)$] We say $N$ is \emph{balanced} if $\iota_N$ is surjective.
\item[$(3)$] We say $N$ is \emph{faithfully balanced} if $\iota_N$ is an
isomorphism.
\end{enumerate}
\end{chunk}

The following definition gives the main object of study in this work.

\begin{definition}\label[definition]{def:endsemidualizing}
We say that $N$ is \emph{$\End$-semidualizing} if
\begin{enumerate}
\item[$(1)$] $\Ext_E^{i}(N,N)=0$ for every $i>0$;
\item[$(2)$] $N$ is faithfully balanced.
\end{enumerate}
We say that $N$ is \emph{faithfully $\End$-semidualizing} if, in addition,
$-\otimes_EN$ detects zero modules: for every right $E$-module $L$,
\[
L\otimes_EN=0\quad\Longrightarrow\quad L=0.
\]
\end{definition}

\begin{remark}
The additional condition in this definition concerns detection of zero
objects, as in the notion of a faithfully semidualizing bimodule in
\cite{HW07}. It does not mean that the tensor functor is faithful on
morphism sets. We will always state the zero-detection condition explicitly.
\end{remark}

We note the following useful dual formulation.

\begin{lem}\label[lem]{bothfunctorsfaithful}
The following conditions are equivalent:
\begin{enumerate}
\item[$(1)$] For every right $E$-module $L$, the equality
$L\otimes_EN=0$ implies $L=0$.
\item[$(2)$] For every left $E$-module $X$, the equality
$\Hom_E(N,X)=0$ implies $X=0$.
\end{enumerate}
\end{lem}

\begin{proof}
Write $(-)^+=\Hom_{\mathbb Z}(-,\mathbb Q/\mathbb Z)$. The character
module functor is exact and detects zero modules. For every right
$E$-module $L$, tensor--Hom adjunction gives
\[
(L\otimes_EN)^+\cong\Hom_E(N,L^+).
\]
Thus $(2)$ implies $(1)$. Since $N$ is finitely presented over $E$, for
every left $E$-module $X$ there is also a natural isomorphism
\[
\Hom_E(N,X)^+\cong X^+\otimes_EN.\qedhere
\]

\end{proof}

Semidualizing modules in the classical sense are clearly
$\End$-semidualizing. The next example shows that $\End$-semidualizing
modules are plentiful.

\begin{example}\label[example]{generatorex}
If $N$ is a generator, then $N$ is $\End$-semidualizing. In particular,
$N:=M\oplus R$ is $\End$-semidualizing for every finitely generated left
$R$-module $M$.
\end{example}

\begin{proof}
Since $R$ is a summand of $N^{\oplus t}$, there are elements
$n_1,\ldots,n_t\in N$ and maps $v_1,\ldots,v_t\in\Hom_R(N,R)$
affording the identity
\[
\sum_{j=1}^t v_j(n_j)=1.
\]
For $x\in N$, define $a_j(x)\in E$ by
$a_j(x)(y)=v_j(y)x$. The maps $a_j:N\to E$ are left $E$-linear, and
the composite
\[
N\xrightarrow{(a_1,\ldots,a_t)} E^t
 \longrightarrow N \qquad \text{given by} \qquad
(f_1,\ldots,f_t)\longmapsto\sum_j f_j(n_j)
\]
is the identity. Thus $N$ is finitely generated projective over $E$, so
$\Ext_E^{>0}(N,N)=0$ and the standing finiteness hypothesis holds.

The generator condition also implies that $N$ is faithful over $R$.
For $s\in\End_E(N)$ and $x\in N$, $E$-linearity gives
\[
s(x)=s\Big(\sum_j a_j(x)(n_j)\Big)
 =\sum_j a_j(x)(s(n_j))
 =\Big(\sum_j v_j(s(n_j))\Big)x.
\]
Hence $s$ is multiplication by an element of $R$, and $\iota_N$ is
surjective as well as injective.
\end{proof}

\begin{example}\label[example]{cogeneratorex}
Let $R$ be a finite-dimensional algebra over a field $k$, and let
$D(-)=\Hom_k(-,k)$. We show that every finite-dimensional cogenerator $N$ is
$\End$-semidualizing. Here the cogenerator condition means that
$D(R_R)$ is a summand of a finite direct sum of copies of $N$.
Indeed, $G=D(N)$ is then a generator as a right $R$-module.
Put $F=\End_{R^{\mathrm{op}}}(G)$ and $E=\End_R(N)$.
The map $e\mapsto D(e)$ identifies $E^{\mathrm{op}}$ with $F$.
By the generator case, $G$ is projective as a left $F$-module
and its double-centralizer map
$R^{\mathrm{op}}\to\End_F(G)$ is an isomorphism. Consequently
$G$ is projective as a right $E$-module, so $N\cong D(G)$ is
injective as a left $E$-module. Moreover, duality identifies
\[
 \End_E(N)\cong\End_F(G)^{\mathrm{op}}\cong R,
\]
and the resulting map from $R$ is its natural action on $N$.
Thus $N$ is faithfully balanced and
$\Ext_E^{>0}(N,N)=0$.

\begin{example}\label[example]{neithergeneratornorcogenerator}
There exist $\End$-semidualizing modules that are neither
generators nor cogenerators and are not classically semidualizing,
even over commutative Artinian local rings.

Let $k$ be a field, write $D(-)=\Hom_k(-,k)$, and set
\[
A=k[x,y]/(x,y)^2,
\qquad
B=k[z,w]/(z,w)^2,
\qquad
R=A\otimes_k B.
\]
Put $I=D(B)$, $G=A\oplus k$, and
\[
N=G\otimes_k I.
\]
Thus $R$ is a commutative Artinian local ring of $k$-dimension
$9$, and $N$ has $k$-dimension $12$.

We first show that $N$ is an $\End$-semidualizing $R$-module. Set
$F=\End_A(G)$. Since $G$ is a generator over $A$,
\Cref{generatorex} and its proof show that $G$ is finitely
generated projective over $F$ and that the homothety map
$A\to\End_F(G)$ is an isomorphism. Also, $I$ is injective
over $B$, and the homothety map $B\to\End_B(I)$ is an
isomorphism since $\End_B(I)\cong\End_B(B)^{\mathrm{op}}\cong B$. The natural tensor--Hom identifications give
\[
E:=\End_R(N)
\cong \End_A(G)\otimes_k\End_B(I)
\cong F\otimes_k B
\]
and
\[
\End_E(N)
\cong \End_F(G)\otimes_k\End_B(I)
\cong A\otimes_k B=R.
\]
The latter identification is induced by the natural action
of $R$ on $N$, so $N$ is faithfully balanced.

Let $P_\bullet\to I$ be a resolution by finitely generated
free $B$-modules. Since $G$ is finitely generated projective
over $F$, the complex $G\otimes_k P_\bullet$ is a resolution
of $N$ by finitely generated projective $E$-modules. Moreover,
\[
\Hom_E(G\otimes_k P_\bullet,N)
\cong
\End_F(G)\otimes_k\Hom_B(P_\bullet,I)
\cong
A\otimes_k\Hom_B(P_\bullet,I).
\]
As $I$ is injective, this complex has zero cohomology in
positive degrees. Hence $\Ext_E^i(N,N)=0$ for every $i>0$,
and $N$ is $\End$-semidualizing.

Next, write
\[
N=C\oplus Q,
\qquad
C=A\otimes_k I,
\qquad
Q=k\otimes_k I.
\]
The natural identifications
\[
\End_R(C)\cong R,
\qquad
\End_R(Q)\cong B
\]
show that $C$ and $Q$ are indecomposable $R$-modules, since $R$ and $B$
are local. The modules $R$ and $D(R)$ are also
indecomposable $R$-modules. Their dimensions and socle dimensions are
\[
\begin{array}{c|cccc}
 & C & Q & R & D(R)\\ \hline
\dim_k(-) & 9 & 3 & 9 & 9\\
\dim_k\operatorname{soc}_R(-) & 2 & 1 & 4 & 1.
\end{array}
\]
Here $\operatorname{soc}_A(A)$ and
$\operatorname{soc}_B(B)$ have dimension $2$, whereas
$\operatorname{soc}_B(I)$ has dimension $1$, and
\[
\operatorname{soc}_R(C)
=
\operatorname{soc}_A(A)\otimes_k\operatorname{soc}_B(I),
\qquad
\operatorname{soc}_R(R)
=
\operatorname{soc}_A(A)\otimes_k\operatorname{soc}_B(B).
\]
Thus neither $C$ nor $Q$ is isomorphic to $R$ or $D(R)$.
By the Krull--Schmidt theorem, neither $R$ nor $D(R)$
is a direct summand of any finite direct sum of copies of
$N$. Therefore $N$ is neither a generator nor a cogenerator.

Finally, tensoring the exact sequence of $A$-modules
\[
0\longrightarrow k
\xrightarrow{\,1\mapsto x\,}A/(y)
\longrightarrow k\longrightarrow0
\]
over $k$ with $I$ gives an exact sequence of $R$-modules
\[
0\longrightarrow Q
\longrightarrow (A/(y))\otimes_k I
\longrightarrow Q\longrightarrow0.
\]
This sequence does not split: $x$ annihilates $Q\oplus Q$,
but acts nontrivially on the middle term. Consequently,
$\Ext_R^1(Q,Q)\ne0$. Since $Q$ is a direct summand of $N$,
we obtain $\Ext_R^1(N,N)\ne0$. Thus $N$ is not classically
semidualizing, and the associated
$(R,E^{\mathrm{op}})$-bimodule is not semidualizing either.
\end{example}

\end{example}

Recall that $R$ is an \emph{Artin algebra} if it is module-finite over
its center and its center is Artinian. Over a self-injective Artin algebra,
the preceding example describes all $\End$-semidualizing modules.

\begin{proposition}\label[proposition]{selfinjectiveendsemi}
Suppose $R$ is a self-injective Artin algebra. Then $N$ is
$\End$-semidualizing if and only if $N$ is a generator. 
\end{proposition}

\begin{proof}
The generator implication is \Cref{generatorex}. Conversely, suppose
$N$ is faithful. Then as $N$ is finitely generated, there is an embedding $R \hookrightarrow N^{\oplus r}$ for some $r$. As $R$ is self-injective, this embedding splits, so $N$ is a generator.\qedhere

\end{proof}

\begin{corollary}
If $R$ is a commutative Artinian Gorenstein ring, then $N$ is
$\End$-semidualizing if and only if $N\cong R\oplus M$ for some
finitely generated $R$-module $M$.
\end{corollary}

\begin{proof}
By \Cref{selfinjectiveendsemi}, it remains to observe that a finitely
generated generator over a commutative Artinian ring contains $R$ as a
summand. Decompose $R$ as a finite product of Artinian local rings $R_1\times\cdots\times R_n$. Let $e_i\in R$ be the idempotent corresponding to $1\in R_i$. Let $N_i=e_iN$, therefore $N\cong\bigoplus_{i=1}^nN_i$. Since $R$ is contained in a direct power of $N$, multiplying by $e_i$ shows that $R_i$ is contained in a direct power of $N_i$.
The Krull--Schmidt theorem shows that $R_i$ is a
summand of $N_i$. Taking the product gives the required summand $R$.
\end{proof}

\begin{remark}\label[remark]{remark:AR}
Over a self-injective Artin algebra, the existence of a semidualizing
bimodule (in the sense of Holm and White \cite{HW07}) which is not a projective generator is equivalent to the existence
of a nonprojective finitely generated module $M$ satisfying
$\Ext_R^{>0}(M,M\oplus R)=0$, that is, a counterexample to the
Auslander--Reiten conjecture for $R$. Indeed, Let $M$ be a nonprojective module such that $\Ext_R^{>0}(M,M\oplus R)=0$ and set $N=M\oplus R$. By \Cref{generatorex}, $N$ is End-semidualizing, but in addition $\Ext_R^{>0}(N,N)=0$, and therefore it is a nonprojective semidualizing bimodule. Conversely, if $N$ is a nonprojective semidualizing bimodule, then in particular it is End-semidualizing and \Cref{generatorex} shows that it is a generator. Moreover, since $\Ext_R^{>0}(N,N)=0$ and $R$ is self-injective, it follows that $\Ext_R^{>0}(N,N\oplus R)=0$. 
\end{remark}

We let $J(T)$ denote the Jacobson radical of a ring $T$. Recall that
$T$ is local if $J(T)$ is a maximal left, equivalently maximal right,
ideal. The following observation will be used to detect the obstruction
to the two-out-of-three property.

\begin{proposition}\label[proposition]{Homfinlength}
Let $T$ be a local ring, let $U$ be a nonzero finitely generated left
$T$-module, and let $X$ be a left $T$-module containing a nonzero
submodule of finite length. Then $\Hom_T(U,X)\ne0$.
In particular, if $R$ is an Artin algebra and $N$ is an indecomposable
nonzero finitely generated left $R$-module, then both $-\otimes_EN$ and
$\Hom_E(N,-)$ detect zero modules.
\end{proposition}

\begin{proof}
By Nakayama's lemma, $U/J(T)U$ is a nonzero vector space over the
division ring $T/J(T)$, so it has a quotient isomorphic to $T/J(T)$.
A nonzero finite length submodule of $X$ has a simple submodule, which
is likewise isomorphic to $T/J(T)$. Composing the resulting surjection
$U\twoheadrightarrow T/J(T)$ with this inclusion into $X$ gives a
nonzero homomorphism.

For the claim about Artin algebras, we note from \cite[Theorem II.2.2(a)]{AR97} that $E$ is local whenever $R$ is an Artin algebra and $N$ is indecomposable. Every nonzero left $E$-module contains a nonzero finite length cyclic
submodule, and $N$ is finitely generated over $E$. Thus the first part
shows that $\Hom_E(N,-)$ detects zero modules, and the tensor assertion
then follows from \Cref{bothfunctorsfaithful}.
\end{proof}

\section{The Auslander Class}\label{sec:auslander}

For a right $R$-module $M$, we let $\epsilon_M^N:M \to \Hom_E(N,M \otimes_R N)$ be the map of right $R$-modules given by $x \mapsto f_x$ where $f_x:N \to M \otimes_R N$ is given by $f_x(y)=x \otimes y$. We just write $\epsilon_M$ if $N$ is understood.

\begin{definition}\label[definition]{def:auslander}
We say a right $R$-module $M$ belongs to the \emph{Auslander class over $R$ with respect to $N$}, denoted by $\A_N(R)$, if the following hold:
\begin{enumerate}
\item[$(1)$] $\Tor_{>0}^{R}(M,N)=0$.
\item[$(2)$] $\Ext_E^{>0}(N,M\otimes_{R} N)=0$.
\item[$(3)$] The map $\epsilon_M^N:M\rightarrow\Hom_E(N,M\otimes_{R} N)$ is an isomorphism.
\end{enumerate}
\end{definition}

\begin{prop}\label[prop]{projectiveex}
The following are equivalent for an $R$-module $N$:
\begin{enumerate}
    \item[$(1)$] $N$ is $\End$-semidualizing.
    \item[$(2)$] Every projective right $R$-module is in $\A_N(R)$.
    \item[$(3)$] Every flat right $R$-module is in $\A_N(R)$. 
\end{enumerate}
\end{prop}
\begin{proof}
For $(2)\Rightarrow(1)$, if every projective right $R$-module is in $\A_N(R)$, then in particular $R$ is, and we immediately obtain that $N$ is $\End$-semidualizing from the definition: $\epsilon_R^N$ identifies with $\iota_N$.

For $(1)\Rightarrow(2)$, first suppose $M=R^{(\Lambda)}$ is free, where $\Lambda$ is any set. The standing hypothesis that ${}_EN$ has a projective resolution with finitely generated terms implies that $\Hom_E(N,-)$ and all $\Ext_E^i(N,-)$ commute with direct sums. Thus
\[
\Hom_E(N,M\otimes_RN)\cong\End_E(N)^{(\Lambda)}
\quad\text{and}\quad
\Ext_E^{>0}(N,M\otimes_RN)=0.
\]
Under this identification, $\epsilon_M^N$ is the direct sum of copies of $\iota_N$, so it is an isomorphism. The $\Tor$ condition is automatic. 

Now let $M$ be an arbitrary projective right $R$-module.
Since $M$ is projective, we automatically have
$\Tor^R_{>0}(M,N)=0$. Choose $L$ so that
$M \oplus L \cong R^{(\Lambda)}$ for some set $\Lambda$.
Then there is an isomorphism
\[
(M \otimes_R N) \oplus (L \otimes_R N)
\cong N^{(\Lambda)}
\]
of left $E$-modules. Since $\Ext_E^{>0}(N,N)=0$,
commutation with direct sums gives
\[
\Ext_E^i(N,N^{(\Lambda)})
\cong \Ext_E^i(N,N)^{(\Lambda)}=0
\qquad\text{for every }i>0.
\]
Taking direct summands, we get
$\Ext_E^{>0}(N,M \otimes_R N)=0$.

Letting $i:M \to R^{(\Lambda)}$ denote the natural split
injection and $p:R^{(\Lambda)}\to M$ the natural split
surjection, we have commutative diagrams:
\[
\begin{tikzcd}[cramped]
    0 & M & {R^{(\Lambda)}} & L & 0 \\
    0 & {\Hom_E(N,M \otimes_R N)}
      & {\Hom_E(N,R^{(\Lambda)} \otimes_R N)}
      & {\Hom_E(N,L \otimes_R N)} & 0
    \arrow[from=1-1, to=1-2]
    \arrow["i", from=1-2, to=1-3]
    \arrow["{\epsilon_M}"', from=1-2, to=2-2]
    \arrow[from=1-3, to=1-4]
    \arrow["{\epsilon_{R^{(\Lambda)}}}", from=1-3, to=2-3]
    \arrow[from=1-4, to=1-5]
    \arrow["{\epsilon_L}", from=1-4, to=2-4]
    \arrow[from=2-1, to=2-2]
    \arrow[from=2-2, to=2-3]
    \arrow[from=2-3, to=2-4]
    \arrow[from=2-4, to=2-5]
\end{tikzcd}
\]
and
\[
\begin{tikzcd}[cramped]
    0 & L & {R^{(\Lambda)}} & M & 0 \\
    0 & {\Hom_E(N,L \otimes_R N)}
      & {\Hom_E(N,R^{(\Lambda)} \otimes_R N)}
      & {\Hom_E(N,M \otimes_R N)} & 0.
    \arrow[from=1-1, to=1-2]
    \arrow[from=1-2, to=1-3]
    \arrow["{\epsilon_L}"', from=1-2, to=2-2]
    \arrow["p", from=1-3, to=1-4]
    \arrow["{\epsilon_{R^{(\Lambda)}}}", from=1-3, to=2-3]
    \arrow[from=1-4, to=1-5]
    \arrow["{\epsilon_M}", from=1-4, to=2-4]
    \arrow[from=2-1, to=2-2]
    \arrow[from=2-2, to=2-3]
    \arrow[from=2-3, to=2-4]
    \arrow[from=2-4, to=2-5]
\end{tikzcd}
\]
The rows are split exact, since both $-\otimes_R N$ and
$\Hom_E(N,-)$ preserve split exact sequences. Since
$\epsilon_{R^{(\Lambda)}}$ is an isomorphism, it follows
from applying the snake lemma to the first diagram that
$\epsilon_M$ is injective, and applying it to the second
gives that $\epsilon_M$ is surjective, so it is an isomorphism.
Therefore $M \in \mathcal{A}_N(R)$.

Note that $(3)\Rightarrow(2)$ is immediate, so we conclude the proof with $(2)\Rightarrow(3)$. Let $M$ be a flat right $R$-module. By Lazard's theorem, we may write $M=\varinjlim_jF_j$ where each $F_j$ is a finitely generated free right $R$-module. Since ${}_EN$ has a projective resolution with finitely generated terms, $\Hom_E(N,-)$ and all $\Ext_E^i(N,-)$ commute with filtered direct limits. Consequently, for $i>0$,
\[
\Ext_E^i(N,M\otimes_RN)
\cong\varinjlim_j\Ext_E^i(N,F_j\otimes_RN)=0.
\]
By naturality of the unit map, $\epsilon_M^N$ identifies with $\varinjlim_j\epsilon_{F_j}^N$, and hence is an isomorphism. Finally, $\Tor_{>0}^R(M,N)=0$ because $M$ is flat.
\end{proof}

For semidualizing modules in the classical sense, the Auslander class enjoys the so-called ``two-out-of-three property" from which many of their key properties can be derived. As we will see, for $\End$-semidualizing modules, subtleties arise due to the noncommutative nature of their definition, and as a consequence the two-out-of-three property will not hold in general. We do however always have the following proposition that shows that it is only one aspect of the two-out-of-three property that may fail in general.

\begin{prop}\label[prop]{auslanderresolving}
Suppose $0 \to A \to B \to C \to 0$ is a short exact sequence of right $R$-modules. If $C \in \A_N(R)$, then $B \in \A_N(R)$ if and only if $A \in \A_N(R)$.
\end{prop}

\begin{proof}
Applying $- \otimes_R N$ to the short exact sequence $0 \to A \to B \to C \to 0$ and considering the associated long exact sequence in $\Tor$, we see that $\Tor_i^R(A,N) \cong \Tor^R_i(B,N)$ for all $i>0$, and that there is a short exact sequence 
\[0 \to A \otimes_R N \to B \otimes_R N \to C \otimes_R N \to 0\]
of left $E$-modules. Applying $\Hom_E(N,-)$ to this sequence and considering the induced long exact sequence in $\Ext$, we see that $\Ext^i_E(N,A \otimes_R N) \cong \Ext^i_E(N,B \otimes_R N)$ for all $i>1$, and that there is an exact sequence 

\[\begin{tikzpicture}[descr/.style={fill=white,inner sep=1.5pt}]
        \matrix (m) [
            matrix of math nodes,
            row sep=1em,
            column sep=1.8em,
            text height=1.5ex, text depth=0.25ex
        ]
        { 0 & \Hom_E(N,A \otimes_R N) & \Hom_E(N,B \otimes_R N) & \Hom_E(N,C \otimes_R N) \\
           \mbox{} & \Ext^1_E(N,A \otimes_R N) & \Ext^1_E(N,B \otimes_R N) & 0 \\
        };
\path[overlay,->, font=\scriptsize,>=latex]
        (m-1-1) edge (m-1-2)
        (m-1-2) edge (m-1-3)
        (m-1-3) edge (m-1-4);
        \path[overlay,->, font=\scriptsize,>=latex]
        (m-1-4) edge[out=355,in=175] (m-2-2)
        (m-2-2) edge (m-2-3)
        (m-2-3) edge (m-2-4);   
\end{tikzpicture}\]

The top row fits in as the bottom row of the commutative diagram:
\[\begin{tikzcd}[cramped]
	0 & A & B & C & 0 \\
	0 & {\Hom_E(N,A \otimes_R N)} & {\Hom_E(N,B \otimes_R N)} & {\Hom_E(N,C \otimes_R N)} & \mbox{}.
	\arrow[from=1-1, to=1-2]
	\arrow[from=1-2, to=1-3]
	\arrow["{\epsilon^N_A}"', from=1-2, to=2-2]
	\arrow[from=1-3, to=1-4]
	\arrow["{\epsilon^N_B}"', from=1-3, to=2-3]
	\arrow[from=1-4, to=1-5]
	\arrow["{\epsilon^N_C}"', from=1-4, to=2-4]
	\arrow[from=2-1, to=2-2]
	\arrow[from=2-2, to=2-3]
	\arrow[from=2-3, to=2-4]
\end{tikzcd}\]

As $C \in \A_N(R)$, we have that $\epsilon^N_C$ is an isomorphism, and it follows that the map \[
\Hom_E(N,B \otimes_R N) \to \Hom_E(N,C \otimes_R N)
\]
is surjective. Thus $\Ext^1_E(N,A \otimes_R N) \cong \Ext^1_E(N,B \otimes_R N)$ as well, and moreover we have from the five lemma that $\epsilon^N_A$ is an isomorphism if and only if $\epsilon^N_B$ is an isomorphism, completing the proof.

\end{proof}

\begin{prop}\label[prop]{auslander2outof3}
Suppose the functor $- \otimes_E N$ detects zero objects. Then $\mathcal{A}_N(R)$ satisfies the two-out-of-three property.
\end{prop}

\begin{proof}
We may suppose $N \ne 0$. From Proposition \ref{auslanderresolving}, it suffices to show that if $0 \to A \to B \to C \to 0$ is a short exact sequence of right $R$-modules with $A,B \in \mathcal{A}_N(R)$, then $C \in \mathcal{A}_N(R)$. Applying $- \otimes_R N$ to the exact sequence above and considering the long exact sequence in $\Tor$, we see that $\Tor^R_{>1}(C,N)=0$, and we have an exact sequence of $E$-modules
\[0 \to \Tor^R_1(C,N) \to A \otimes_R N \to B \otimes_R N \to C \otimes_R N \to 0.\]
Applying $\Hom_E(N,-)$, we get an exact sequence
\[0 \to \Hom_E(N,\Tor^R_1(C,N)) \to \Hom_E(N,A \otimes_R N) \to \Hom_E(N,B \otimes_R N)\]
There is a commutative diagram:
\[\begin{tikzcd}[cramped]
	A & B \\
	{\Hom_E(N,A \otimes_R N)} & {\Hom_E(N,B \otimes_R N)}
	\arrow[from=1-1, to=1-2]
	\arrow["{\epsilon_A}"', from=1-1, to=2-1]
	\arrow["{\epsilon_B}"', from=1-2, to=2-2]
	\arrow[from=2-1, to=2-2]
\end{tikzcd}\]
whose vertical arrows are isomorphisms, from which we see that the map $\Hom_E(N,A \otimes_R N) \to \Hom_E(N,B \otimes_R N)$ is injective. Thus $\Hom_E(N,\Tor^R_1(C,N))=0$. From \Cref{bothfunctorsfaithful}, we have that $\Hom_E(N,-)$ detects zero objects, and it follows that $\Tor^R_1(C,N)=0$. We thus have a short exact sequence 
\[0 \to A \otimes_R N \to B \otimes_R N \to C \otimes_R N \to 0.\]
Applying $\Hom_E(N,-)$ to this exact sequence and considering the long exact sequence in $\Ext$, we see that $\Ext^{>0}_E(N,C \otimes_R N)=0$ and that there is a short exact sequence fitting in as the bottom row of the commutative diagram:
\[\begin{tikzcd}[cramped]
	0 & A & B & C & 0 \\
	0 & {\Hom_E(N,A \otimes_R N)} & {\Hom_E(N,B \otimes_R N)} & {\Hom_E(N,C \otimes_R N)} & 0.
	\arrow[from=1-1, to=1-2]
	\arrow[from=1-2, to=1-3]
	\arrow["{\epsilon_A}"', from=1-2, to=2-2]
	\arrow[from=1-3, to=1-4]
	\arrow["{\epsilon_B}"', from=1-3, to=2-3]
	\arrow[from=1-4, to=1-5]
	\arrow["{\epsilon_C}"', from=1-4, to=2-4]
	\arrow[from=2-1, to=2-2]
	\arrow[from=2-2, to=2-3]
	\arrow[from=2-3, to=2-4]
	\arrow[from=2-4, to=2-5]
\end{tikzcd}\]

As $A,B \in \mathcal{A}_N(R)$, $\epsilon_A$ and $\epsilon_B$ are isomorphisms, and then so is $\epsilon_C$ by the five lemma, so $C \in \mathcal{A}_N(R)$, as claimed.
\end{proof}

\begin{corollary}\label[corollary]{cor:exactSeq}
Suppose the functor $- \otimes_E N$ detects zero objects. Consider the following exact sequence of right $R$-modules
\[
0\rightarrow M_1\rightarrow M_2\rightarrow\cdots\rightarrow M_n\rightarrow0.
\]
If $M_i\in\mathcal{A}_N(R)$ for all $i\neq j$, then $M_j\in \mathcal{A}_N(R)$.
\end{corollary}

\begin{proof}
By induction on $n$, using \Cref{auslander2outof3} and \Cref{auslanderresolving}. 
\end{proof}

\begin{corollary}\label[corollary]{cor:auslanderFlatDim}
Suppose the functor $- \otimes_E N$ detects zero objects. The following conditions are equivalent:
\begin{enumerate}
\item The $R$-module $N$ is End-semidualizing.
\item The class $\mathcal{A}_N(R)$ contains every right $R$-module of finite flat dimension.
\end{enumerate}
\end{corollary}
\begin{proof}
It follows immediately from \Cref{cor:exactSeq} and \Cref{projectiveex}.
\end{proof}

 The next result applies \Cref{Homfinlength} to the Artin algebra setting:

\begin{proposition}\label[proposition]{artinauslanderclass}
Suppose $R$ is an Artin algebra and that $N$ is an indecomposable $R$-module. Then $\A_N(R)$ satisfies the $2$-out-$3$ property. Moreover, $N$ is End-semidualizing if and only if $\A_N(R)$ contains every right $R$-module of finite flat dimension.
\end{proposition}
\begin{proof}
It follows from \cite[Theorem II.2.2(a)]{AR97} that $E$ is local, and then \Cref{Homfinlength} implies that the functor $- \otimes_E N$ detects zero objects. The result then follows immediately from \Cref{auslander2outof3}, and \Cref{cor:auslanderFlatDim}.
\end{proof}

\section{The Bass Class}\label{sec:bass}
We recall for any $R$-module $M$, that the trace map $\delta^N_M:\Hom_R(N,M) \otimes_E N \to M$ is given by $f \otimes x \mapsto f(x)$. We write $\delta_M$ if $N$ is understood.
\begin{defn}\label[defn]{def:bass}
We say a left $R$-module $M$ belongs to the \emph{Bass class} over $R$ with respect to $N$, denoted by $\mathcal{B}_N(R)$, if the following hold:
\begin{enumerate}
\item[$(1)$] $\Ext^{>0}_R(N,M)=0$.
\item[$(2)$] $\Tor^E_{>0}(\Hom_R(N,M),N)=0$.
\item[$(3)$] The trace map $\delta^N_M:\Hom_R(N,M) \otimes_E N \to M$ is an isomorphism. 
\end{enumerate}
\end{defn}

\begin{proposition}\label[proposition]{injectiveex}
The following are equivalent for an $R$-module $N$:
\begin{enumerate}
\item $N$ is $\End$-semidualizing.
\item Every injective left $R$-module is in $\mathcal{B}_N(R)$.
\end{enumerate}
\end{proposition}

\begin{proof}
Let $I$ be an injective left $R$-module, and let $P\to N$ be a resolution of the left $E$-module $N$ by finitely generated projective left $E$-modules. The $R$-action on $N$ makes $\Hom_E(P,N)$ a complex of left $R$-modules. For each finitely generated projective left $E$-module $Q$, there is a natural evaluation isomorphism of abelian groups
\[
\Hom_R(N,I)\otimes_EQ
\xrightarrow{\ \cong\ }
\Hom_R(\Hom_E(Q,N),I),
\qquad
f\otimes q\longmapsto\bigl(h\mapsto f(h(q))\bigr).
\]
This is immediate for finite free modules, and then follows for their direct summands. Applying it degreewise to $P$, and using the exactness of $\Hom_R(-,I)$, gives natural isomorphisms
\begin{equation}\label{eq:injective-evaluation}
\Tor_i^E(\Hom_R(N,I),N)
\cong\Hom_R(\Ext_E^i(N,N),I)
\qquad(i\geq0).
\end{equation}
For $i=0$, the isomorphism is the map
\[
\theta_I:\Hom_R(N,I)\otimes_EN\longrightarrow
\Hom_R(\End_E(N),I),\qquad
\theta_I(f\otimes x)(s)=f(s(x)).
\]
Moreover, we have the following commutative diagram of abelian groups:
\[
\begin{tikzcd}[column sep=large]
\Hom_R(N,I)\otimes_EN \arrow[r,"\theta_I"]\arrow[d,"\delta_I^N"']
& \Hom_R(\End_E(N),I)\arrow[d,"{\Hom_R(\iota_N,I)}"]\\
I\arrow[r,"\gamma_I"']&\Hom_R(R,I),
\end{tikzcd}
\]
where $\gamma_I$ denotes the natural isomorphism.

Now suppose every injective left $R$-module belongs to $\B_N(R)$. By \eqref{eq:injective-evaluation},
$\Hom_R(\Ext_E^i(N,N),I)=0$ for every $i>0$ and every injective $I$. Taking $I$ to be an injective cogenerator shows that $\Ext_E^{>0}(N,N)=0$. Since $\theta_I$ and $\delta_I^N$ are isomorphisms, the diagram shows that $\Hom_R(\iota_N,I)$ is an isomorphism for every injective $I$. If $I$ is an injective cogenerator, then $\iota_N$ is an isomorphism. Thus $N$ is $\End$-semidualizing.

Conversely, suppose $N$ is $\End$-semidualizing and fix an injective left $R$-module $I$. The vanishing of $\Tor_{>0}^E(\Hom_R(N,I),N)$ follows from \eqref{eq:injective-evaluation}, and $\Ext_R^{>0}(N,I)=0$ is automatic. Since $\iota_N$ is an isomorphism, the diagram shows that $\delta_I^N$ is an isomorphism. Thus $I\in\B_N(R)$.
\end{proof}

\begin{proposition}\label[proposition]{morita-classes}
The following conditions are equivalent:
\begin{enumerate}
\item $N$ is a finitely generated projective generator over $R$.
\item $\A_N(R)$ consists of all right $R$-modules.
\item $\B_N(R)$ consists of all left $R$-modules.
\end{enumerate}
\end{proposition}
\begin{proof}
Suppose $(2)$ holds. The vanishing of $\Tor_1^R(M,N)$ for every right $R$-module $M$ implies that $N$ is flat. Since $N$ is finitely generated and flat over the Noetherian ring $R$, it is projective. Let $T$ be its trace ideal, namely the two-sided ideal of $R$ consisting of finite sums of values $f(n)$ with $f\in\Hom_R(N,R)$ and $n\in N$. 

Since $N$ is finitely generated projective, choose $R$-linear maps
\[
N \xrightarrow{u} R^t \xrightarrow{v} N
\qquad\text{with}\qquad vu=\operatorname{id}_N.
\]
Let $e_1,\ldots,e_t$ be the standard basis of $R^t$, and write
$u(n)=(f_1(n),\ldots,f_t(n))$, where each
$f_j:N\to R$ is $R$-linear. By the definition of $T$, we have
$f_j(n)\in T$ for every $n\in N$. Hence
\[
n=v(u(n))=\sum_{j=1}^t f_j(n)v(e_j)\in TN.
\]
Thus $N\subseteq TN$, and the reverse inclusion is automatic.
Therefore $TN=N$.
Thus $(R/T)\otimes_RN=0$, and the isomorphism $\epsilon_{R/T}^N$ forces $R/T=0$. Writing $1$ as a finite sum of trace values gives a surjection from a finite direct sum of copies of $N$ onto $R$. This surjection splits, and thus $N$ is a generator.

Suppose $(3)$ holds. The vanishing of $\Ext_R^1(N,M)$ for every left $R$-module $M$ implies that $N$ is projective. Surjectivity of $\delta_R^N$ again writes $1$ as a finite sum of trace values, and shows that $N$ is a generator.

Suppose $(1)$ holds. Since $N$ is a generator, a surjection
$N^{\oplus t}\to R$ shows that its trace ideal is $R$.
By \cite[Theorem~A.2(c),(g) and Proposition~A.3]{AG60},
$N$ is a finitely generated projective generator over $E$,
and the homothety map $R\to\Hom_E(N,N)$ is an isomorphism.
Applying \cite[Proposition~A.4]{AG60} over $R$ and over $E$,
and using the tensor--Hom adjunctions, shows that the maps
\[
\epsilon_M^N:
M\longrightarrow\Hom_E(N,M\otimes_R N)
\quad\text{and}\quad
\delta_X^N:
\Hom_R(N,X)\otimes_E N\longrightarrow X
\]
are isomorphisms for every right $R$-module $M$ and every
left $R$-module $X$.

Since $N$ is projective over both $R$ and $E$, we also have the required Ext and Tor vanishing, proving
$(1)\Rightarrow(2),(3)$.
\end{proof}

\begin{prop}\label[prop]{bassclass-ses}
Suppose $0 \to A \to B \to C \to 0$ is a short exact sequence of left $R$-modules. If $A \in \B_N(R)$, then $B \in \B_N(R)$ if and only if $C \in \B_N(R)$.

\end{prop}

\begin{proof}
Suppose $A \in \B_N(R)$. Then applying $\Hom_R(N,-)$ to the exact sequence $0 \to A \to B \to C \to 0$, we see from the associated long exact sequence in $\Ext$ that $\Ext^{i}_R(N,B) \cong \Ext^i_R(N,C)$ for all $i>0$ and that there is a short exact sequence of right $E$-modules 
\[0 \to \Hom_R(N,A) \to \Hom_R(N,B) \to \Hom_R(N,C) \to 0.\]
Applying $- \otimes_E N$ to this sequence and considering the induced long exact sequence in $\Tor$ we see that $\Tor^E_i(\Hom_R(N,B),N) \cong \Tor^E_i(\Hom_R(N,C),N)$ for all $i>1$ and that there is an exact sequence of the form

\[\begin{tikzpicture}[descr/.style={fill=white,inner sep=1.5pt}]
        \matrix (m) [
            matrix of math nodes,
            row sep=1em,
            column sep=1.8em,
            text height=1.5ex, text depth=0.25ex
        ]
        { \Hom_R(N,A) \otimes_E N & \Hom_R(N,B) \otimes_E N & \Hom_R(N,C) \otimes_E N & 0 \\
           0 & \Tor^E_1(\Hom_R(N,B),N) & \Tor^E_1(\Hom_R(N,C),N) & \mbox{} \\
        };
\path[overlay,->, font=\scriptsize,>=latex]
        (m-1-1) edge (m-1-2)
        (m-1-2) edge (m-1-3)
        (m-1-3) edge (m-1-4);
        \path[overlay,->, font=\scriptsize,>=latex]
        (m-2-3) edge[out=365,in=185] (m-1-1)
        (m-2-2) edge (m-2-3)
        (m-2-1) edge (m-2-2);   
\end{tikzpicture}\]

the top row of which fits into the commutative diagram
\[\begin{tikzcd}[cramped]
	& {\Hom_R(N,A) \otimes_E N} & {\Hom_R(N,B) \otimes_E N} & {\Hom_R(N,C) \otimes_E N} & 0 \\
	0 & A & B & C & 0.
	\arrow[from=1-2, to=1-3]
	\arrow["{\delta^N_A}"', from=1-2, to=2-2]
	\arrow[from=1-3, to=1-4]
	\arrow["{\delta^N_B}"', from=1-3, to=2-3]
	\arrow[from=1-4, to=1-5]
	\arrow["{\delta^N_C}"', from=1-4, to=2-4]
	\arrow[from=2-1, to=2-2]
	\arrow[from=2-2, to=2-3]
	\arrow[from=2-3, to=2-4]
	\arrow[from=2-4, to=2-5]
\end{tikzcd}\]
As $A \in \B_N(R)$, we have that $\delta^N_A$ is an isomorphism, and it follows the map $\Hom_R(N,A) \otimes_E N \to \Hom_R(N,B) \otimes_E N$ is injective, so $\Tor^E_1(\Hom_R(N,B),N) \cong \Tor^E_1(\Hom_R(N,C),N)$. Moreover, we have from the five lemma, that $\delta^N_B$ is an isomorphism if and only if $\delta^N_C$ is an isomorphism, completing the proof. 
\end{proof}

\begin{prop}\label[prop]{prp:Bass2out3}
Suppose the functor $- \otimes_E N$ detects zero objects. Then $\B_N(R)$ satisfies the two-out-of-three property. 
\end{prop}

\begin{proof}
From \Cref{bassclass-ses}, it suffices to show that if there is an exact sequence $0\rightarrow A\rightarrow B\rightarrow C\rightarrow 0$ with $B,C \in \B_N(R)$, then $A \in \B_N(R)$. Applying $\Hom_R(N,-)$ to the exact sequence $0 \to A \to B \to C \to 0$, and considering the induced long exact sequence in $\Ext$, we have $\Ext^{>1}_R(N,A)=0$ and that there is an exact sequence
\[0 \to \Hom_R(N,A) \to \Hom_R(N,B) \to \Hom_R(N,C) \to \Ext^1_R(N,A) \to 0.\]
Applying $- \otimes_E N$ yields an exact sequence 
\[\Hom_R(N,B) \otimes_E N \to \Hom_R(N,C) \otimes_E N \to \Ext^1_R(N,A) \otimes_E N \to 0.\]
On the other hand, there is a commutative diagram:
\[\begin{tikzcd}[cramped]
	{\Hom_R(N,B) \otimes_E N} & {\Hom_R(N,C) \otimes_E N} \\
	B & C
	\arrow[from=1-1, to=1-2]
	\arrow["{\delta^N_B}"', from=1-1, to=2-1]
	\arrow["{\delta^N_C}"', from=1-2, to=2-2]
	\arrow[from=2-1, to=2-2]
\end{tikzcd}\]
whose vertical arrows are isomorphisms, since $B,C \in \B_N(R)$, and whose bottom arrow is a surjection. In particular, we have that the map $\Hom_R(N,B) \otimes_E N \to \Hom_R(N,C) \otimes_E N$ is surjective, so $\Ext^1_R(N,A) \otimes_E N=0$. As the functor $- \otimes_E N$ detects zero objects, it follows that $\Ext^1_R(N,A)=0$. 

Thus there is a short exact sequence
\[0 \to \Hom_R(N,A) \to \Hom_R(N,B) \to \Hom_R(N,C) \to 0.\]
Applying $-\otimes_E N$ to this sequence and considering the induced long exact sequence in $\Tor$, we see that $\Tor^E_{>0}(\Hom_R(N,A),N)=0$ and that there is a short exact sequence fitting as the top row of the commutative diagram:
\[\begin{tikzcd}[cramped]
	0 & {\Hom_R(N,A) \otimes_E N} & {\Hom_R(N,B) \otimes_E N} & {\Hom_R(N,C) \otimes_E N} & 0 \\
	0 & A & B & C & 0.
	\arrow[from=1-1, to=1-2]
	\arrow[from=1-2, to=1-3]
	\arrow["{\delta^N_A}"', from=1-2, to=2-2]
	\arrow[from=1-3, to=1-4]
	\arrow["{\delta^N_B}"', from=1-3, to=2-3]
	\arrow[from=1-4, to=1-5]
	\arrow["{\delta^N_C}"', from=1-4, to=2-4]
	\arrow[from=2-1, to=2-2]
	\arrow[from=2-2, to=2-3]
	\arrow[from=2-3, to=2-4]
	\arrow[from=2-4, to=2-5]
\end{tikzcd}\]
As $\delta^N_B$ and $\delta^N_C$ are isomorphisms, it follows that $\delta^N_A$ is an isomorphism, so $A \in \B_N(R)$, as claimed.    
\end{proof}

\begin{corollary}\label[corollary]{cor:exactSeqBass}
Suppose the functor $- \otimes_E N$ detects zero objects and consider the following exact sequence of left $R$-modules
\[
0\rightarrow M_1\rightarrow M_2\rightarrow\cdots\rightarrow M_n\rightarrow0.
\]
If $M_i\in\mathcal{B}_N(R)$ for all $i\neq j$, then $M_j\in \mathcal{B}_N(R)$.
\end{corollary}

\begin{proof}
By induction on $n$, using \Cref{bassclass-ses} and \Cref{prp:Bass2out3}. 
\end{proof}

\begin{corollary}\label[corollary]{cor:Bassinjdim}
Suppose the functor $- \otimes_E N$ detects zero objects. The following conditions are equivalent:
\begin{enumerate}
\item The $R$-module $N$ is End-semidualizing.
\item The class $\mathcal{B}_N(R)$ contains every left $R$-module of finite injective dimension.
\end{enumerate}
\end{corollary}
\begin{proof}
It follows immediately from \Cref{cor:exactSeqBass,injectiveex}.
\end{proof}

The next result applies the previous results of this section to the Artin algebra setting:

\begin{proposition}\label[proposition]{artinbassclass}
Suppose $R$ is an Artin algebra and that $N$ is an indecomposable $R$-module. Then $\B_N(R)$ satisfies the $2$-out-$3$ property. Moreover, $N$ is End-semidualizing if and only if $\B_N(R)$ contains every left $R$-module of finite injective dimension.
\end{proposition}
\begin{proof}
It follows from \cite[Theorem II.2.2(a)]{AR97} that $E$ is local, and then \Cref{Homfinlength} implies that the functor $- \otimes_E N$ detects zero objects. The result then follows immediately from \Cref{prp:Bass2out3}, and \Cref{cor:Bassinjdim}.
\end{proof}

\begin{example}\label[example]{artintwooutofthreefails}
The two-out-of-three property can fail even over an Artin algebra.
Let $k$ be a field and let
\[
 R=\left\{\begin{pmatrix}a&b\\0&c\end{pmatrix}:a,b,c\in k\right\}.
\]
Write $e_1$ for the matrix with $a=1,b=c=0$, $e_2$ for the matrix $c=1,a=b=0$, and $u$ for the matrix $a=c=0,b=1$. Since $ue_2=u$, we can define the quotient $S=Re_2/ku$, where we have used $ku=Ru$. We also set $N=R\oplus S$. By
Example \ref{generatorex}, $N$ is $\End$-semidualizing. There is
an exact sequence of right $R$-modules
\[
 0\longrightarrow e_2R\xrightarrow{\alpha}e_1R
 \longrightarrow T\longrightarrow0,
 \qquad \alpha(e_2)=u,\quad T=e_1R/ku.
\]
Both $e_2R$ and $e_1R$ are projective since $R=e_1R\oplus e_2R$ and belong to $\A_N(R)$ by \Cref{projectiveex}.
However, as $k$-vector spaces one has $Re_2\cong ku\oplus ke_2$, and since $e_2^2=e_2$ and $e_1e_2=0$, it follows that $e_2S\cong k$ and $e_1S=0$, so tensoring the sequence
with $S$ gives
\[
\Tor_1^R(T,S)\cong k,
\]
therefore $\Tor_1^R(T,N)\cong k$. Thus $T\notin\A_N(R)$.

The same example gives a failure for the Bass class. Let
$D(-)=\Hom_k(-,k)$. Dualizing gives an exact sequence
\[
 0\longrightarrow D(T)\longrightarrow D(e_1R)
 \longrightarrow D(e_2R)\longrightarrow0.
\]
The last two modules are injective left $R$-modules and hence
belong to $\B_N(R)$, whereas
$\Ext_R^1(N,D(T))\cong D(\Tor_1^R(T,N))\ne0$.
\end{example}

\section{Foxby Equivalence}\label{sec:foxby}

When $N$ is faithfully balanced, we identify $\End_E(N)$ with $R$.
The notation $\B_N(E)$ then means the class defined in
Definition~\ref{def:bass} with the roles of $R$ and $E$ interchanged.
Thus its objects are left $E$-modules $X$ satisfying
$\Ext_E^{>0}(N,X)=0$, $\Tor_{>0}^R(\Hom_E(N,X),N)=0$, and
the evaluation isomorphism $\Hom_E(N,X)\otimes_R N\cong X$.
This use of the definition does not require $E$ to be Noetherian.

\begin{lem}\label[lem]{splitting}
We have the following:
\begin{enumerate}
\item[$(1)$] For any right $R$-module $M$, the map
$\delta^N_{M\otimes_RN}$ is a split surjection of left
$E$-modules, with section $\epsilon_M^N\otimes_R\id_N$.
\item[$(2)$] For any left $E$-module $M$, the map
$\epsilon^N_{\Hom_E(N,M)}$ is a split injection of right
$R$-modules, with retraction $\Hom_E(N,\delta_M^N)$.
\end{enumerate}
\end{lem}

\begin{proof}
$(1)$ follows as, for any elementary tensor
$m\otimes n\in M\otimes_RN$, we have
\[
\begin{aligned}
\delta^N_{M\otimes_RN}
(\epsilon_M^N\otimes_R\id_N)(m\otimes n)
&=\delta^N_{M\otimes_RN}(\epsilon_M^N(m)\otimes n)\\
&=\epsilon_M^N(m)(n)=m\otimes n.
\end{aligned}
\]

For $(2)$, note that
$\epsilon^N_{\Hom_E(N,M)}(g)=f_g$, where
$f_g:N\to\Hom_E(N,M)\otimes_RN$ is given by
$f_g(y)=g\otimes y$. So
$\Hom_E(N,\delta_M^N)(f_g)=\delta_M^N\circ f_g$.
But then for any $y\in N$, we have
\[
\delta_M^N(f_g(y))=\delta_M^N(g\otimes y)=g(y),
\]
and it follows that $\delta_M^N\circ f_g=g$,
giving the claim.
\end{proof}

\begin{theorem}\label[theorem]{auslanderbassclassforward}
Suppose $N$ is a faithfully balanced left $R$-module. Then we have the following:
\begin{enumerate}
\item[$(1)$] If a right $R$-module $M$ is in $\mathcal{A}_N(R)$, then $M \otimes_R N$ is in $\mathcal{B}_N(E)$. 
\item[$(2)$] If a left $E$-module $M$ is in $\mathcal{B}_N(E)$, then $\Hom_E(N,M)$ is in $\A_N(R)$.
\item[$(3)$] The functors $- \otimes_R N:\A_N(R) \to \mathcal{B}_N(E)$ and $\Hom_E(N,-):\mathcal{B}_N(E) \to \A_N(R)$ are equivalences. 
\end{enumerate}
\end{theorem}

\begin{proof}
We first prove $(1)$. Let $M\in\A_N(R)$. We first notice that $\Ext_E^{>0}(N,M\otimes_RN)=0$ since $M$ is in $\A_N(R)$, therefore $M\otimes_RN$ satisfies \Cref{def:bass}(1). Since $N$ is faithfully balanced $\mathrm{End}_E(N)=R$, therefore to show that $M\otimes_RN$ satisfies \Cref{def:bass}(2) we need to show that 
\[\Tor_{>0}^R(\Hom_E(N,M\otimes_RN),N)=0.\]
Since $M\in\A_N(R)$ it follows that $\Hom_E(N,M\otimes_RN)\cong M$ as $R$-modules and the desired vanishing now follows again since $M\in\A_N(R)$. Finally we need to show that the trace map
\[
\delta^N_{M\otimes_RN}:\Hom_E(N,M\otimes_RN)\otimes_RN\rightarrow M\otimes_RN
\]
is an isomorphism. But as $M \in \A_N(R)$, we have that $\epsilon^N_M$ is an isomorphism, and then so is $\epsilon^N_M \otimes \id_N$. By \Cref{splitting}, $\epsilon^N_M \otimes \id_N$ is a section for $\delta^N_{M \otimes_R N}$, and it follows that $\delta^N_{M \otimes_R N}$ is an isomorphism.

We now show $(2)$. Let $M \in \B_N(E)$. That $\Tor^R_{>0}(\Hom_E(N,M),N)=0$ follows from the definition of the Bass class and the fact that $R \cong \End_E(N)$. We moreover have that $\Ext^{>0}_E(N,M) \cong \Ext^{>0}_E(N,\Hom_E(N,M) \otimes_R N)=0$.
To see that $\epsilon^N_{\Hom_E(N,M)}$ is an isomorphism, we have that $\delta^N_M$ is an isomorphism as $M \in \B_N(E)$, and then so is $\Hom_E(N,\delta^N_M)$ which is a retraction for $\epsilon^N_{\Hom_E(N,M)}$ by \Cref{splitting}. It follows that $\epsilon^N_{\Hom_E(N,M)}$ is an isomorphism, so $\Hom_E(N,M) \in \A_N(R)$.

For $(3)$, the maps $\epsilon_M$ induce an isomorphism of functors $\id_{\A_N(R)} \cong \Hom_E(N,- \otimes_R N)$, while the maps $\delta^N_M$ induce an isomorphism of functors $\id_{\B_N(E)} \cong \Hom_E(N,-) \otimes_R N$, so we have the desired equivalences.
\end{proof}

\begin{theorem}\label[theorem]{foxbyrecognition}
Suppose $N$ is a faithfully balanced left $R$-module. 
\begin{enumerate}
\item[$(1)$] If the functor $- \otimes_R N$ detects zero objects, then a right $R$-module $M$ is in $\mathcal{A}_N(R)$ if and only if $M \otimes_R N$ is in $\mathcal{B}_N(E)$. 
\item[$(2)$] If the functor $- \otimes_E N$ detects zero objects, then a left $E$-module $M$ is in $\mathcal{B}_N(E)$ if and only if $\Hom_E(N,M)$ is in $\A_N(R)$. 
\end{enumerate}

\end{theorem}

\begin{proof}
The forward directions of (1) and (2) follow immediately from \Cref{auslanderbassclassforward}. For the converse in $(1)$, suppose $M \otimes_R N \in \mathcal{B}_N(E)$. Then $\delta^N_{M \otimes_R N}$ is an isomorphism, but from \Cref{splitting}, it is a split surjection with section $\epsilon^N_M \otimes \id_N$. In particular, $\epsilon^N_M \otimes \id_N$ is an isomorphism. But we have $\coker(\epsilon^N_M \otimes \id_N) \cong \coker(\epsilon^N_M) \otimes_R N=0$, and it follows since $- \otimes_R N$ detects zero objects that $\epsilon^N_M$ is surjective. So there is a short exact sequence of right $R$-modules
\[0 \rightarrow \ker(\epsilon^N_M) \rightarrow M \xrightarrow{\epsilon^N_M} \Hom_E(N,M \otimes_R N) \rightarrow 0.\]
But as $M \otimes_R N \in \B_N(E)$, we have in particular that $\Tor^R_1(\Hom_E(N,M \otimes_R N),N)=0$, so applying $- \otimes_R N$ to the sequence above reveals that $\ker(\epsilon^N_M \otimes \id_N) \cong \ker(\epsilon^N_M) \otimes_R N=0$. It follows since $- \otimes_R N$ detects zero objects that $\ker(\epsilon^N_M)=0$ so that $\epsilon^N_M$ is an isomorphism. But then for all $i>0$ we have $\Tor^R_i(\Hom_E(N,M \otimes_R N),N) \cong \Tor^R_i(M,N)=0$, while the condition $\Ext^{>0}_E(N,M \otimes_R N)=0$ follows at once from the definition of Bass class. Therefore $M \in \A_N(R)$.

We now prove the converse in $(2)$. Suppose $\Hom_E(N,M) \in \A_N(R)$. Then $\epsilon^N_{\Hom_E(N,M)}$ is an isomorphism, and it follows from \Cref{splitting} that so is its retraction $\Hom_E(N,\delta^N_M)$. But $\ker(\Hom_E(N,\delta^N_M)) \cong \Hom_E(N,\ker(\delta^N_M))=0$, and it follows from \Cref{bothfunctorsfaithful} that $\delta^N_M$ is injective. There is thus a short exact sequence 
\[0 \rightarrow \Hom_E(N,M) \otimes_R N \xrightarrow{\delta^N_M} M \rightarrow \coker(\delta^N_M) \rightarrow 0\]
As $\Hom_E(N,M) \in \A_N(R)$, we in particular have that $\Ext^1_E(N,\Hom_E(N,M) \otimes_R N)=0$, so applying $\Hom_E(N,-)$ to this sequence reveals that $\coker(\Hom_E(N,\delta^N_M)) \cong \Hom_E(N,\coker(\delta^N_M))=0$. It follows from \Cref{bothfunctorsfaithful} that $\coker(\delta^N_M)=0$, so $\delta^N_M$ is an isomorphism. Then for all $i>0$, we have $\Ext^i_E(N,\Hom_E(N,M) \otimes_R N) \cong \Ext^i_E(N,M)=0$. That $\Tor^R_{>0}(\Hom_E(N,M),N)=0$ follows at once from the definition of Auslander class, so $M \in \B_N(E)$.
\end{proof}

\begin{proposition}\label[proposition]{character-foxby-duality}
Write $(-)^+=\Hom_{\mathbb Z}(-,\mathbb Q/\mathbb Z)$.
For every right $R$-module $M$ and
every left $R$-module $L$ one has
\[
M\in\A_N(R)\quad\Longleftrightarrow\quad M^+\in\B_N(R),
\qquad
L\in\B_N(R)\quad\Longleftrightarrow\quad L^+\in\A_N(R).
\]
In particular, $\A_N(R)$ has the two-out-of-three property if and
only if $\B_N(R)$ does.
\end{proposition}

\begin{proof}
The character module functor is exact, detects zero modules, and hence
reflects isomorphisms. The standing hypotheses provide resolutions of
$N$ by finitely generated projective modules over both $R$ and $E$.
Tensor--Hom adjunction and evaluation on the terms of these resolutions
give, for $i\geq0$,
\begin{align*}
\Ext_R^i(N,M^+)&\cong\Tor_i^R(M,N)^+,\\
\Tor_i^E(X^+,N)&\cong\Ext_E^i(N,X)^+,\\
\Tor_i^R(L^+,N)&\cong\Ext_R^i(N,L)^+,\\
\Ext_E^i(N,Y^+)&\cong\Tor_i^E(Y,N)^+,
\end{align*}
where $X$ is a left $E$-module and $Y$ is a right $E$-module.
The second and third identities use the finite generation of the
projective terms, while the first and fourth follow directly from adjunction.
In degree zero there are natural isomorphisms
\[
\Hom_R(N,M^+)\cong(M\otimes_RN)^+,
\qquad
\Hom_R(N,L)^+\cong L^+\otimes_RN.
\]
The displayed identities identify the two vanishing conditions for
$M\in\A_N(R)$ with those for $M^+\in\B_N(R)$, and those for
$L\in\B_N(R)$ with those for $L^+\in\A_N(R)$.

We check the natural maps as well. Set $X=M\otimes_RN$. The evaluation
isomorphism
\[
X^+\otimes_EN\xrightarrow{\ \cong\ }\Hom_E(N,X)^+,
\qquad
\varphi\otimes n\longmapsto\bigl(h\mapsto\varphi(h(n))\bigr),
\]
identifies the dual $(\epsilon_M^N)^+$ with $\delta_{M^+}^N$.
Indeed, under $X^+\cong\Hom_R(N,M^+)$, a functional $\varphi$
corresponds to $f_\varphi$ with
$f_\varphi(n)(m)=\varphi(m\otimes n)$, and both maps send
$\varphi\otimes n$ to the functional
$m\mapsto\varphi(m\otimes n)$.

Similarly, put $Y=\Hom_R(N,L)$. Under
$Y^+\cong L^+\otimes_RN$ and
$(Y\otimes_EN)^+\cong\Hom_E(N,Y^+)$, the dual
$(\delta_L^N)^+$ identifies with $\epsilon_{L^+}^N$:
for $\lambda\in L^+$, $n\in N$, and $g\in Y$, both send
$\lambda$ to the map whose value at $n$, evaluated on $g$, is
$\lambda(g(n))$. Since character duality reflects vanishing and
isomorphisms, all three defining conditions are equivalent in both
assertions.

Finally, character duality takes a short exact sequence to a short exact
sequence in the opposite module category, with its order reversed.
Applying the established equivalences to its three terms transfers the
two-out-of-three property from either class to the other.
\end{proof}

\section{The Commutative Case}\label{sec:commutative}
Throughout this section we suppose $R$ is a commutative Noetherian ring. The module $N$ is finitely generated, but the modules in the Auslander and Bass classes need not be finitely generated. In this setting $E=\End_R(N)$ is a module-finite $R$-algebra and hence is left and right Noetherian. Moreover, $N$ is finitely generated over $E$: any set of $R$-generators is also a set of $E$-generators. Thus $N$ has resolutions by finitely generated free modules over both $R$ and $E$.

We first note that
\begin{equation}\label{center-doublecentralizer}
Z(E)=\End_E(N).
\end{equation}
Indeed, if $f\in Z(E)$, then for every $e\in E$ and $n\in $N,
\[
f(en)=(fe)(n)=(ef)(n)=ef(n),
\]
therefore $f$ is $E$-linear, giving $Z(E)\subseteq\End_E(N)$. For the opposite inclusion we first note that if $r\in R$ and $\lambda_r:N\rightarrow N$ is the map $n\mapsto rn$, then $\lambda_R\in E$ by the commutativity of $R$. Therefore if $f\in \End_E(N)$, then
\[
f(rn)=f(\lambda_r(n))=\lambda_rf(n)=rf(n),
\]
showing that $f$ is $R$-linear. Moreover, for all $e\in E$, the $E$-linearity of $f$ implies
\[
f(e(n))=ef(n),
\]
and therefore $f\in Z(E)$.

\begin{prop}\label[prop]{extunderflatbasechange}
Let $T$ be a left Noetherian $R$-algebra and $S$ a commutative flat $R$-algebra. Let $A$ be a finitely generated left $T$-module and let $B$ be a left $T$-module. Then the natural map
\[
\theta_{A,B}:\Hom_T(A,B)\otimes_R S
\longrightarrow\Hom_{T\otimes_R S}(A\otimes_R S,B\otimes_R S)
\]
given by $f\otimes s\mapsto(a\otimes t\mapsto f(a)\otimes st)$ is an isomorphism. Moreover, it induces isomorphisms
\[
\Ext_T^i(A,B)\otimes_R S
\cong\Ext_{T\otimes_R S}^i(A\otimes_R S,B\otimes_R S)
\qquad(i\geq0).
\]
\end{prop}

\begin{proof}
We first note that the claim is clear if $A$ is finitely generated free over $T$. Let $F$ be a resolution of $A$ by finitely generated free $T$-modules. Since $S$ is flat over $R$, the complex $F\otimes_R S$ is a free resolution of $A\otimes_R S$ over $T\otimes_R S$. Moreover, the degreewise natural maps give an isomorphism of complexes
\[
\Hom_T(F,B)\otimes_R S
\cong\Hom_{T\otimes_R S}(F\otimes_R S,B\otimes_R S).
\]
Taking cohomology, and again using flatness, gives the claim.
\end{proof}

\begin{cor}
Suppose $S$ is a commutative flat $R$-algebra. Then
\[
Z(E)\otimes_R S\cong Z(\End_S(N\otimes_R S)).
\]
\end{cor}

\begin{proof}
From \Cref{extunderflatbasechange}, we have
$E\otimes_R S\cong\End_S(N\otimes_R S)$. Applying the same proposition over $E$ and using \eqref{center-doublecentralizer} gives the claim.
\end{proof}

\begin{prop}\label[prop]{endflatbasechange}
Suppose $S$ is a commutative Noetherian flat $R$-algebra. If $N$ is $\End$-semidualizing over $R$, then $N\otimes_R S$ is $\End$-semidualizing over $S$. The converse holds if $S$ is faithfully flat. In particular, the End-semidualizing condition behaves well with respect to localization and completion.
\end{prop}

\begin{proof}
Write $N_S=N\otimes_R S$ and $E_S=E\otimes_R S$. By \Cref{extunderflatbasechange}, $E_S\cong\End_S(N_S)$ and
\[
\Ext_E^i(N,N)\otimes_R S\cong\Ext_{E_S}^i(N_S,N_S)
\qquad(i\geq0).
\]
There is a commutative diagram
\[
\begin{tikzcd}
R\otimes_R S \arrow[r,"\iota_N\otimes S"] \arrow[d,"\cong"'] &
\End_E(N)\otimes_R S \arrow[d,"\theta_{N,N}"]\\
S \arrow[r,"\iota_{N_S}"'] & \End_{E_S}(N_S).
\end{tikzcd}
\]
The vertical maps are isomorphisms. Thus the vanishing of positive self-extensions and the double-centralizer isomorphism ascend along $R\to S$. If $S$ is faithfully flat, the same formulas show that they descend.
\end{proof}

\begin{prop}\label[prop]{foxbyflatbasechange}
Suppose $S$ is a commutative Noetherian flat $R$-algebra and let $M$ be an $R$-module.
\begin{enumerate}
\item If $M\in\A_N(R)$, then $M\otimes_R S\in\A_{N\otimes_R S}(S)$. The converse holds if $S$ is faithfully flat.
\item If $M\in\B_N(R)$, then $M\otimes_R S\in\B_{N\otimes_R S}(S)$. The converse holds if $S$ is faithfully flat.
\end{enumerate}
\end{prop}

\begin{proof}
Put $E_S=E\otimes_R S\cong\End_S(N\otimes_R S)$. Flat base change for $\Tor$, and \Cref{extunderflatbasechange} applied over $R$ and $E$, give
\begin{align*}
\Tor_i^R(M,N)\otimes_R S
&\cong\Tor_i^S(M\otimes_R S,N\otimes_R S),\\
\Ext_E^i(N,M\otimes_R N)\otimes_R S
&\cong\Ext_{E_S}^i(N\otimes_R S,(M\otimes_R N)\otimes_R S),\\
\Ext_R^i(N,M)\otimes_R S
&\cong\Ext_S^i(N\otimes_R S,M\otimes_R S),\\
\Tor_i^E(\Hom_R(N,M),N)\otimes_R S
&\cong\Tor_i^{E_S}(\Hom_S(N\otimes_R S,M\otimes_R S),N\otimes_R S).
\end{align*}
The $\Tor$ identities follow by tensoring free resolutions with $S$; flatness over $R$ preserves their exactness. Under the degree-zero identifications, the maps $\epsilon_M\otimes_R S$ and $\delta_M^N\otimes_R S$ are precisely the corresponding unit and evaluation maps after base change. This proves ascent. Faithful flatness reflects the required vanishings and isomorphisms, proving descent.
\end{proof}

\begin{prop}\label[prop]{auslanderbassduality}
Suppose $I$ is an injective $R$-module and set $D(-)=\Hom_R(-,I)$. Then the following hold for every $R$-module $M$.
\begin{enumerate}
\item[$(1)$] If $M\in\A_N(R)$, then $D(M)\in\B_N(R)$.
\item[$(2)$] If $M\in\B_N(R)$, then $D(M)\in\A_N(R)$.
\end{enumerate}
Both converses hold if $I$ is an injective cogenerator.
\end{prop}

\begin{proof}
The functor $D$ is exact. Using degreewise finitely generated free resolutions of $N$ over $R$ and $E$ gives natural isomorphisms
\begin{align}
D\Tor_i^R(M,N)&\cong\Ext_R^i(N,D(M)),\label{duality-tor-R}\\
D\Ext_E^i(N,X)&\cong\Tor_i^E(D(X),N),\label{duality-ext-E}\\
D\Ext_R^i(N,M)&\cong\Tor_i^R(D(M),N),\label{duality-ext-R}\\
D\Tor_i^E(Y,N)&\cong\Ext_E^i(N,D(Y)).\label{duality-tor-E}
\end{align}
Here $X$ is a left $E$-module and $Y$ is a right $E$-module, endowed with their induced $R$-module structures. For example, the finite freeness of each term $P_j$ of an $E$-free resolution of $N$ gives
$D\Hom_E(P_j,X)\cong D(X)\otimes_E P_j$. Exactness of $D$ then gives \eqref{duality-ext-E}. The other isomorphisms follow in the same way from tensor--Hom adjunction.

In degree zero these formulas include
\[
D(M\otimes_R N)\cong\Hom_R(N,D(M)),
\qquad
D\Hom_R(N,M)\cong D(M)\otimes_R N.
\]
Under the degree-zero instance of \eqref{duality-ext-E} and the first of these isomorphisms, the dual of the unit $\epsilon_M$ is the evaluation map $\delta_{D(M)}^N$. Similarly, the dual of $\delta_M^N$ is identified with $\epsilon_{D(M)}$.

Consequently, the duals of the three conditions defining $\A_N(R)$ are exactly the three conditions defining membership of $D(M)$ in $\B_N(R)$. The same argument applies with the two classes interchanged. This proves the implications. If $I$ is a cogenerator, $D$ detects zero modules and, being exact, detects isomorphisms. It therefore also proves both converses.
\end{proof}

\begin{remark}
When $(R,\m,k)$ is local, one may take $I=E_R(k)$, so Matlis duality exchanges membership in the two classes. In particular, when $R$ is complete it gives the usual duality between the corresponding finitely generated and Artinian modules.
\end{remark}

\begin{comment}
\begin{cor}\label[cor]{commutative-twooutofthree-duality}
The class $\A_N(R)$ has the two-out-of-three property if and only if $\B_N(R)$ has the two-out-of-three property.
\end{cor}

\begin{proof}
Choose an injective cogenerator $I$, for example
$I=\bigoplus_{\m\in\Max(R)}E_R(R/\m)$. Applying $D=\Hom_R(-,I)$ to a short exact sequence gives another short exact sequence, with its order reversed. Now apply \Cref{auslanderbassduality} and its converses.
\end{proof}
\end{comment}

\begin{defn}
Inspired by the terminology in \cite{DE21}, we say an $R$-module $N$ is \emph{locally deep} if
$\depth_{R_{\p}}N_{\p}\geq\depth R_{\p}$ for all $\p\in\Spec(R)$, where the depth of the zero module is $\infty$.
\end{defn}

\begin{theorem}\label[theorem]{pdmcm}
Suppose $N$ is an $\End$-semidualizing module. Then the following are equivalent:
\begin{enumerate}
\item[$(1)$] $N$ is locally deep.
\item[$(2)$] Every $R$-module $M$ with $\fd_R(M)<\infty$ belongs to $\A_N(R)$.
\item[$(3)$] Every $R$-module $M$ with $\id_R(M)<\infty$ belongs to $\B_N(R)$.
\end{enumerate}
\end{theorem}

\begin{proof}
For $(1)\Rightarrow(2)$, first observe that
\begin{equation}\label{deep-tor-vanishing}
\Tor_{>0}^R(M,N)=0\qquad\text{whenever }\fd_R(M)<\infty.
\end{equation}
Indeed, this can be checked locally. Over each localization at a maximal ideal, the dependency formula \cite[Theorem 3.5]{ciFlat} gives
\[
\sup\{i\mid\Tor_i^R(M,N)\neq0\}
=\sup_{\p\in\Supp M\cap\Supp N}
\{\depth R_{\p}-\depth_{R_{\p}}M_{\p}-\depth_{R_{\p}}N_{\p}\}.
\]
Here the displayed formula is applied over that local ring. Its hypotheses hold because $M$ has finite flat dimension, so the derived tensor product is bounded and $M$ has finite complete intersection flat dimension. Local deepness therefore makes the right-hand side at most zero, proving \eqref{deep-tor-vanishing}.

We now proceed by induction on $\fd_R(M)$. The flat case was established in \Cref{projectiveex}. Suppose $\fd_R(M)>0$ and consider an exact sequence
\[
0\longrightarrow L\longrightarrow F\longrightarrow M\longrightarrow0
\]
where $F$ is flat and $\fd_R(L)=\fd_R(M)-1$. By \eqref{deep-tor-vanishing}, tensoring with $N$ gives a short exact sequence
\[
0\longrightarrow L\otimes_R N\longrightarrow F\otimes_R N
\longrightarrow M\otimes_R N\longrightarrow0.
\]
The induction hypothesis gives $L,F\in\A_N(R)$. The long exact sequence in $\Ext_E(N,-)$ therefore shows that
$\Ext_E^{>0}(N,M\otimes_R N)=0$ and that applying $\Hom_E(N,-)$ to the displayed sequence is exact. Naturality gives a map from the original short exact sequence to this last one, whose components are $\epsilon_L,\epsilon_F,\epsilon_M$. Since the first two components are isomorphisms, so is $\epsilon_M$. Thus $M\in\A_N(R)$.

For $(2)\Rightarrow(1)$, choose $\p\in\Spec(R)$ and a maximal $R_{\p}$-regular sequence $\underline{x}$ in $\p R_{\p}$. Set $M=R_{\p}/\underline{x}R_{\p}$. The Koszul resolution of $M$ over $R_{\p}$ is a finite flat resolution over $R$, so $\fd_R M<\infty$. By assumption $M\in\A_N(R)$, and hence
\[
\Tor_{>0}^{R_{\p}}(R_{\p}/\underline{x}R_{\p},N_{\p})
\cong\Tor_{>0}^R(M,N)_{\p}=0.
\]
These groups are the positive Koszul homology of $\underline{x}$ on the finite module $N_{\p}$. Their vanishing implies that $\underline{x}$ is $N_{\p}$-regular, unless $N_{\p}=0$, in which case the desired depth inequality is automatic.

For $(2)\Leftrightarrow(3)$, choose an injective cogenerator $I$ and write $D=\Hom_R(-,I)$. The functor $D$ sends flat modules to injective modules and injective modules to flat modules. Exactness of $D$ consequently sends a finite injective resolution to a finite flat resolution, and a finite flat resolution to a finite injective resolution.

If $(2)$ holds and $\id_R(M)<\infty$, then $D(M)$ has finite flat dimension and belongs to $\A_N(R)$. The converse in \Cref{auslanderbassduality}(2) yields $M\in\B_N(R)$. The reverse implication follows in the same way from \Cref{auslanderbassduality}(1).
\end{proof}

\begin{cor}\label[cor]{twooutofthree-necessity}
Suppose $N$ is End-semidualizing. If either $\A_N(R)$ or $\B_N(R)$
has the two-out-of-three property, then $N$ is locally deep.
\end{cor}
\begin{proof}
By \Cref{character-foxby-duality}, it suffices to assume that $\A_N(R)$ has the two-out-of-three property. The class $\A_N(R)$ contains every flat module
by \Cref{projectiveex}. Induction on flat dimension therefore shows that
every module of finite flat dimension is in $\A_N(R)$, and
\Cref{pdmcm} gives the conclusion.
\end{proof}

\begin{theorem}\label[theorem]{2outof3commutative}
Suppose $(R,\m)$ is a commutative Noetherian local ring and that $N$ is an $\End$-semidualizing module. Suppose additionally that one of the following holds:
\begin{enumerate}
\item[$(1)$] $\End_R(N)$ is local and $N$ is locally free on the punctured spectrum of $R$.
\item[$(2)$] $\End_{R_{\p}}(N_{\p})$ is local for every $\p\in\Spec(R)$.
\end{enumerate}
Then $\A_N(R)$ and $\B_N(R)$ have the two-out-of-three property.
\end{theorem}

\begin{proof}
Let $0\to A\to B\to C\to0$ be a short exact sequence with $A,B\in\A_N(R)$. By \Cref{auslanderresolving}, it suffices to show that $C\in\A_N(R)$. Applying $-\otimes_R N$ gives an exact sequence of left $E$-modules
\[
0\longrightarrow T\longrightarrow A\otimes_R N
\longrightarrow B\otimes_R N\longrightarrow C\otimes_R N\longrightarrow0,
\qquad T=\Tor_1^R(C,N),
\]
and $\Tor_{>1}^R(C,N)=0$. Applying $\Hom_E(N,-)$ to the beginning of this sequence and using the natural isomorphisms $\epsilon_A$ and $\epsilon_B$ shows that
\begin{equation}\label{obstruction-Hom-zero}
\Hom_E(N,T)=0.
\end{equation}

Under the assumptions of $(1)$, $\Tor^R_1(C,N)$ contains a finite length $R$-submodule, which is thus a finite length $E$-module, since a composition series over $E$ is a composition series over $R$. It then follows from Lemma \Cref{Homfinlength} that $T=0$.

Under $(2)$, if $T\neq0$, choose a nonzero cyclic $E$-submodule $X\subseteq T$ and a minimal prime $\p$ of its $R$-support. The module $X_{\p}$ is nonzero and has finite length over $R_{\p}$, hence also over $E_{\p}$. Since $E_{\p}\cong\End_{R_{\p}}(N_{\p})$ is local, \Cref{Homfinlength} gives
$\Hom_{E_{\p}}(N_{\p},T_{\p})\neq0$. This contradicts \eqref{obstruction-Hom-zero}, since \Cref{extunderflatbasechange} identifies that module with $\Hom_E(N,T)_{\p}$.

Thus in either case $T=0$. Tensoring the original short exact sequence with $N$ is now exact. The long exact $\Ext_E$ sequence and the naturality of the units give $\Ext_E^{>0}(N,C\otimes_R N)=0$ and show that $\epsilon_C$ is an isomorphism, as in the proof of \Cref{auslander2outof3}. Hence $C\in\A_N(R)$. The assertion for the Bass class follows from \Cref{character-foxby-duality}.
\end{proof}

\begin{remark}\label[remark]{henselian-caution}
A finite indecomposable module over a Henselian local ring has a local endomorphism ring; see \cite[Proposition 1.18]{Yoshino90}. Thus hypothesis $(1)$ of \Cref{2outof3commutative} holds whenever $R$ is Henselian and $N$ is indecomposable and locally free on the punctured spectrum. 
\end{remark}

\begin{cor}\label[cor]{enddualizingismcm}
Suppose $R$ is a commutative Cohen--Macaulay local ring and $N$ is an $\End$-semidualizing module. If either hypothesis of \Cref{2outof3commutative} holds, then $N$ is maximal Cohen--Macaulay. In particular, this is the case if $R$ is Henselian and $N$ is indecomposable and locally free on the punctured spectrum.
\end{cor}

\begin{proof}
The theorem gives $\depth_R N\geq\depth R=\dim R$. As $N$ is finite and nonzero, $\depth_R N\leq\dim R$, so equality holds.
\end{proof}

\begin{example}\label[example]{notwooutofthree}
Suppose $(R,\m,k)$ is a commutative Cohen--Macaulay local ring of positive dimension and set $N=k\oplus R$. As noted in Example \ref{generatorex}, $N$ is $\End$-semidualizing. However, $N$ has depth zero and is not maximal Cohen--Macaulay. By \Cref{pdmcm}, $\A_N(R)$ does not contain every module of finite flat dimension, although it contains every flat module by \Cref{projectiveex}. Hence $\A_N(R)$ does not satisfy the two-out-of-three property. By \Cref{character-foxby-duality}, neither does $\B_N(R)$.
\end{example}

In fact the concrete obstruction in Example \ref{notwooutofthree} is given by the following.

\begin{example}
Let $(R,\m,k)$ be a commutative Cohen--Macaulay local ring of positive dimension, set $N=k\oplus R$, and let $x\in\m$ be a nonzerodivisor on $R$. Then
\[
\Hom_E(N,\Tor_1^R(R/xR,N))=0,
\qquad
\Tor_1^R(R/xR,N)\neq0.
\]
Thus in general the vanishing in \eqref{obstruction-Hom-zero} does not force the Tor module to vanish.
\end{example}

\begin{proof}
Applying $-\otimes_R N$ to $0\to R\xrightarrow{x}R\to R/xR\to0$ gives an exact sequence of left $E$-modules
\[
0\longrightarrow\Tor_1^R(R/xR,N)\longrightarrow N\xrightarrow{x}N.
\]
Applying $\Hom_E(N,-)$ gives
\[
0\longrightarrow\Hom_E(N,\Tor_1^R(R/xR,N))
\longrightarrow\End_E(N)\xrightarrow{x}\End_E(N).
\]
The homothety identifies the last map with multiplication by $x$ on $R$, which is injective. This proves the first assertion. On the other hand,
$\Tor_1^R(R/xR,N)=(0:_N x)\cong k$, proving the second. Equivalently, both copies of $R$ in the original short exact sequence belong to $\A_N(R)$, whereas $R/xR$ does not.
\end{proof}

We conclude this section by studying the behavior of End-semidualizing modules modulo regular sequences.

\begin{lem}\label[lem]{endregularsequence}
Suppose $N$ is an $R$-module with $\Ext^i_R(N,N)=0$ for $1\leq i\leq n$.
If $\underline{x}=x_1,\dots,x_n$ is an $R$-regular and $N$-regular sequence,
then reduction modulo $\underline{x}$ induces an isomorphism of
$R/\underline{x}R$-algebras
\[
 E/\underline{x}E\xrightarrow{\ \cong\ }
 \End_{R/\underline{x}R}(N/\underline{x}N).
\]
\end{lem}
\begin{proof}
We first treat one element $x=x_1$, writing $\overline R=R/xR$
and $\overline N=N/xN$. Applying $\Hom_R(N,-)$ to
$0\to N\xrightarrow{x}N\to\overline N\to0$ gives
\[
 0\longrightarrow E\xrightarrow{x}E\longrightarrow
 \Hom_R(N,\overline N)\longrightarrow\Ext_R^1(N,N)=0.
\]
Since $\Hom_R(N,\overline N)=\End_{\overline R}(\overline N)$, the asserted isomorphism follows
when $n=1$.

For induction, let $F$ be a free resolution of $N$ over $R$.
Since $x_1$ is regular on both $R$ and $N$, we have
$\Tor^R_i(R/x_1R,N)=0$ for every $i>0$. Thus $F/x_1F$ is a free
resolution of $\overline N=N/x_1N$ over $\overline R=R/x_1R$, and
\[
 \Ext^i_{\overline R}(\overline N,\overline N)
 \cong\Ext^i_R(N,\overline N).
\]
The long exact sequence obtained from $0\to N\xrightarrow{x_1}N\to\overline N\to0$ gives $\Ext^i_R(N,\overline N)=0$ for $1\leq i<n$.
The sequence $x_2,\dots,x_n$ is regular on $\overline R$ and
$\overline N$, so induction completes the proof.
\end{proof}

\begin{proposition}\label[proposition]{endsemidualizingregularsequence}
Let $N$ be an $\End$-semidualizing module. If
$\Ext_R^i(N,N)=0$ for $1\leq i\leq n$ and
$\underline{x}=x_1,\ldots,x_n$ is an $R$-regular and $N$-regular
sequence, then $N/\underline{x}N$ is an $\End$-semidualizing module
over $R/\underline{x}R$.
\end{proposition}
\begin{proof}
The proof of Lemma \ref{endregularsequence} shows that the required
range of self-$\Ext$ vanishing decreases by one after quotienting
by each element. It therefore suffices to prove the assertion for
a single element $x$. Set $\overline R=R/xR$, $\overline N=N/xN$,
and $\overline E=E/xE$. By Lemma \ref{endregularsequence},
$\overline E=\End_{\overline R}(\overline N)$.

The image of $x$ is central in $E$ and is a nonzerodivisor on $E$:
if $xf=0$ for $f\in E$, then $xf(n)=0$ for every $n\in N$,
and the $N$-regularity of $x$ gives $f=0$. Consequently
$0\to E\xrightarrow{x}E\to\overline E\to0$ is a free right
$E$-resolution of $\overline E$. Since $x$ is also regular on
$N$, it follows that $\Tor^E_i(\overline E,N)=0$ for $i>0$.
Reducing a free left $E$-resolution of $N$ modulo $x$ therefore
gives a free $\overline E$-resolution of $\overline N$ and yields
\[
 \Ext^i_{\overline E}(\overline N,\overline N)
 \cong\Ext^i_E(N,\overline N)\qquad(i\geq0).
\]
Applying $\Hom_E(N,-)$ to
$0\to N\xrightarrow{x}N\to\overline N\to0$ and using
$\Ext_E^{>0}(N,N)=0$ now gives
\[
 \Ext^{>0}_{\overline E}(\overline N,\overline N)=0,
 \qquad
 \End_{\overline E}(\overline N)
 \cong\End_E(N)/x\End_E(N)\cong R/xR.
\]
These identifications send multiplication by $r+xR$ to the
endomorphism it induces on $\overline N$. Thus the last
isomorphism is the homothety map, proving that $\overline N$
is $\End$-semidualizing.
\end{proof}

\section*{Acknowledgements}

This material is based upon work supported by the National Science Foundation under Grant No. DMS-1928930 and by the Alfred P. Sloan Foundation under grant G-2021-16778, while the second author was in residence at the Simons Laufer Mathematical Sciences Institute (formerly MSRI) in Berkeley, California, during the Spring 2024 semester. The first author was partly supported by the Simons Foundation grant MPS-TSM-00007849.

\section*{AI Use Disclosure}

Various AI models were used to check for inconsistencies and to improve the exposition. AI models are also responsible for providing us with Examples \ref{neithergeneratornorcogenerator} and \ref{artintwooutofthreefails}. All other content is due to the authors.

\end{document}